\documentclass[12pt,oneside,reqno]{amsart}
\usepackage{amssymb,amsmath}
\usepackage[english]{babel}
\usepackage{amsfonts}
\usepackage{color}
\usepackage{amsthm}
\usepackage[colorlinks=true,linkcolor=NavyBlue, citecolor=NavyBlue,urlcolor=NavyBlue]{hyperref}
\usepackage{graphicx}
\usepackage{mathtools}
\usepackage[usenames,dvipsnames]{pstricks}
\usepackage{bm}
\usepackage{amsmath}
\usepackage{listings}
\usepackage{hyperref}
\usepackage[shortlabels]{enumitem}

\usepackage{xy}
\usepackage[draft]{fixme}
\let\phi\varphi

\newcommand{\normal}{\mathrel{\triangleleft}}

\theoremstyle{plain}
\newtheorem{theorem}{Theorem}[section]

\newtheorem{lemma}[theorem]{Lemma}

\newtheorem{corollary}[theorem]{Corollary}

\theoremstyle{remark}

\theoremstyle{definition}
\newtheorem{definition}[theorem]{Definition}

\newtheorem*{notation*}{Notation}

\usepackage{listings}

\usepackage[T1]{fontenc}
\usepackage[variablett]{lmodern}
\usepackage{xcolor}
\newenvironment{nouppercase}{%
  \renewcommand{\uppercasenonmath}[1]{}}{}
  \usepackage{collectbox}
\makeatletter

\begin{document}
\title[]{\Large{The derangements subgroup in a finite permutation group and the Frobenius--Wielandt Theorem}}
\author[R.~A.~Bailey]{R.~A.~Bailey}
\author[P.~J.~Cameron]{P.~J.~Cameron}
\author[N.~Gavioli]{N.~Gavioli}
\author[C.~M.~Scoppola]{C.~M.~Scoppola}

\address{School of Mathematics and Statistics\\
University of St Andrews\\
Mathematical Institute\\
North Haugh\\
St Andrews KY16 9SS\\
United Kingdom}
\address{DISIM \\
 Universit\`a degli Studi dell'Aquila\\
 via Vetoio\\
 I-67100 Coppito (AQ)\\
 Italy}       
\email[R.A.~Bailey]{rab24@st-andrews.ac.uk}
\email[P.J.~Cameron]{pjc20@st-andrews.ac.uk} 
\email[C.M.~Scoppola]{scoppola@univaq.it}
\email[N.~Gavioli]{norberto.gavioli@univaq.it}

\date{} \thanks{The third and fourth authors are members of INdAM - GNSAGA}.

\begin{abstract}
It is known that if the derangements subgroup of a transitive non-regular permutation group is a proper subgroup, then it is a Frobenius--Wielandt kernel, and, conversely, minimal Frobenius--Wielandt kernels are proper derangements subgroups. We present here a short survey of the literature on this topic,  and we show that, although there are no restrictions on the structure of the $p$-groups appearing as Frobenius--Wielandt complements, a $p$-group appears as a one-point stabiliser in a transitive non-regular permutation group with a proper derangements subgroup if and only if it satisfies a certain group-theoretic condition.
\end{abstract}
\begin{nouppercase}
\maketitle
\end{nouppercase}

\section{Introduction}
 
In 1895, Maillet~\cite{maillet} suggested the study of finite transitive
permutation groups in which any two-point stabiliser is trivial. The major
result on these groups was proved in 1901, with Frobenius' famous
theorem~\cite{frobenius}, asserting that such a group has a regular normal
subgroup, using his recently developed theory of group characters.

Frobenius' Theorem can be stated in two equivalent but different-looking ways, as follows (see~\cite[Satz V.7.6, V.8.2]{huppert}).

\begin{theorem}\label{frob} 
\ 
\begin{itemize}
\item[(a)] If a finite group $G$ has a non-trivial proper subgroup $H$ such that
$H\cap H^g=1$ for all $g\notin H$, then $N = G\setminus\bigcup_{g \in G}(H^g\setminus\{1\})$ is a normal subgroup such 
that $NH=G$ and $N\cap H=\{1\}$.
\item[(b)] Let $G$ be a finite transitive permutation group of degree greater
than~$1$; suppose that a point stabiliser $H$ is non-trivial but any
two-point stabiliser is trivial. Then $G$ has a regular normal subgroup $N$.
\end{itemize}
\end{theorem}

In such a group, $H$ is called a \emph{Frobenius complement} and $N$ the
\emph{Frobenius kernel}.

Note that the transitive permutation action of $G$ on the cosets of $H$, under the hypothesis $H\cap H^g=1$ for all $g\notin H$, is necessarily faithful.

Later, Thompson~\cite{thompson} showed that the Frobenius kernel is nilpotent,
and Zassenhaus~\cite{zassenhaus} determined the possible structures of
Frobenius complements: in particular, each Sylow $p$-subgroup has a single subgroup of order $p$. 

The Frobenius groups can be characterized as follows 
(see~\cite[Satz V.8.5]{huppert}).

\begin{theorem}\label{fpf}
The finite group $G$ is a Frobenius group with Frobenius complement $H$ and Frobenius kernel $N$ if and only if $G$ is a semidirect product $N \rtimes H$ in which the (proper) subgroup $H$ acts fixed-point-freely by conjugation on the  normal subgroup $N$, meaning that for $1 \neq h \in H$ and $1\neq n \in N$,  $n \neq h^{-1}nh$ holds. 
\end{theorem}

Equivalently, the Frobenius complements are characterised (see \cite{loupass}) as the finite groups acting fixed-point-freely on an absolutely irreducible module in characteristic~$0$ or coprime to $|G|$.\label{fpf_action}

We make a couple of remarks.
\begin{itemize}
\item
The theorem does not hold for infinite groups. (A Tarski monster,
see~\cite{olshanskii}, acting by
conjugation on its subgroups of prime order, is a counterexample.) It has
been suggested that the term ``Maillet group'' be used for a group which
satisfies the hypotheses of Frobenius' theorem (Theorem~\ref{frob}(b))
(apart possibly from finiteness), and a ``Frobenius group''
for one satisfying the conclusion.
\item
There is until now no proof of the theorem without using character theory,
though various special cases can be proved (for example, when $|H|$ is
even, or when $G$ is doubly transitive, see~\cite{flavell}); see also Terry Tao's interesting argument~\cite{tao}.
\end{itemize}

The two forms of the theorem led to different developments for abstract
groups and for permutation groups. We report on these developments and how they interplay in the next section.
In the final section,
we explain the construction of examples, and conclude with a characterisation of the $p$-groups that can appear as one-point stabilisers in a transitive non-regular permutation group that has a proper derangements subgroup.

\section{A short survey of the existing literature}

Notation varies a lot in the abundant existing literature. We fix here our own notation, essentially the same as in \cite{kaplan}.

 \begin{definition}\label{5subsets}
 Let $G$ be a finite group, and $1<H<G$. Set $n = |G:H|$.
 \begin{itemize}\label{notation}
 \item  $\Delta = \Delta_G(H) := G \setminus \bigcup_{g \in G} H^g$. 
 \item $D = D_G(H) := \langle\Delta\rangle$. Note that $D$ is a normal subgroup of $G$.
 \item $U = U_G(H):= \langle H \cap H^g \mid g \in G \setminus H\rangle$. Note that $U$ is a normal subgroup of $H$.
 \item $W = W_G(H):= U^G$, the normal closure of $U$ in $G$. 
 \item $K = K_G(H) = G \setminus \bigcup_{g \in G} (H \setminus U)^g$.
 \end{itemize}
 \end{definition}
 
 \medskip
 
 We record here some elementary well-known facts about the subsets of $G$ in 
Definition~\ref{5subsets}. The reader is referred to the cited papers for the complete proofs. Here we will only give references and comments. 
 
 \begin{lemma}\label{elem}
Let the finite group $G$ have a non-trivial proper subgroup $H$, and consider the transitive action of $G$ by right multiplication on the set of right cosets of $H$.
With the above notation, we have the following.
\begin{enumerate}[(a)]
\item \label{item:one} $D$ is a transitive subgroup of $G$;
\item \label{item:two} $W \leq D$;
\item \label{item:three} $\Delta \subset  K \subseteq D$;
\item \label{item:four} $U \leq W\cap H \leq D\cap H$.
\end{enumerate}
\end{lemma}

\begin{proof}
A very short and elegant proof of \ref{item:one} and \ref{item:two} is contained in  \cite{zantema}. Also \cite{knappschmid}, \cite{kaplan} provide different proofs of the same claims. Note that \ref{item:three} and \ref{item:four} follow directly from the definitions via \ref{item:two}. 
\end{proof}

In the permutation group case, $H$ is the stabiliser of one point in the
action of $G$ on the set of right cosets of $H$, while $\Delta$ is the set of \emph{derangements}
(permutations with no fixed points) in $G$, and $D$ is the derangements subgroup,
see~\cite{bcgr}. Moreover, $W$ is the group generated by the elements of $G$
with more than one fixed point.

On the abstract side, Wielandt~\cite{wielandt} proved in 1958 the following
elegant generalisation of Theorem~\ref{frob}. The notation $A\normal B$ means that $A$ is a \emph{proper} normal subgroup of~$B$.

\begin{theorem}\label{wie}
Let $H$ be a proper subgroup of the finite group $G$, and let $H^* \normal H$  such  that, for all $g \in G\setminus H$, $H \cap H^g \leq H^*$. Put
$K^* = G\setminus\bigcup_{g\in G} (H\setminus H^*)^g$. Then $K^*$ is a 
normal subgroup of $G$ such that $K^*\cap H = H^*$ and $G = HK^*$. 
\end{theorem}

Different proofs of Theoreom~\ref{wie} can be found in Huppert's book~\cite[Satz V.7.5]{huppert}, in \cite{knappschmid} and in \cite{kaplan}. In the third of these references, a group $G$ satisfying the hypotheses of Theorem~\ref{wie}
is called a \emph{Frobenius--Wielandt group} (FW group for short) with \emph{FW complement}
$H/H^*$ and \emph {FW kernel} $K^*$. All these proofs 
involve, to some extent, character theory.

Note that, in our notation, the hypothesis of Theorem~\ref{wie} can be written as
\[U = U_G(H) \leq H^* \normal H,\]
and therefore the possible FW complements for $G$ are all the nontrivial factor groups of $H/U$. (Thus, $H/U$ is the largest FW complement.)	
In particular, if $H^* = U$ then $K^* = K$, showing the following result.

\begin{corollary}\label{corwie} 
With the notation as in Definition~\ref{5subsets},
$K$ is a normal subgroup of $G$, and $K \cap H = U$. 
\end{corollary}

Now the inclusions in Lemma~\ref{elem}\ref{item:three},\ref{item:four}, 
together with the result of Corollary~\ref{corwie}, give the following result,
given in \cite{knappschmid} and in \cite{kaplan}.

\begin{corollary}
With the notation of Definition~\ref{5subsets},
\begin{enumerate}[(a)]
\item \label{item:uno} $D = K$; 
\item \label{item:due} $U = W\cap H = D\cap H$.
\end{enumerate}
\end{corollary}

Note that, while Frobenius groups are permutation groups, the permutation action of a FW group $G$ on the cosets of its subgroup $H$ is not necessarily faithful. The corollaries above provide, in this case, the information analogous to that offered by the second form of Frobenius' theorem (Theorem~\ref{frob}(b)).

The results above show that FW groups and groups in which $D_G(H) < G$ (or, equivalently, $U_G(H) < H$) are essentially the same class. The only difference between the two concepts rests on the fact that the definition of a FW complement is closed under taking factors: if $H/H^*$ is a FW complement for $G$, and $H^* \leq N \normal H$ then $H/N$ is a FW complement for $G$. Then $D_G(H) \cap H = U_G(H)$ turns out to be the smallest normal subgroup $H^*$ of $H$ such that $H/H^*$ is a FW complement; equivalently, $K$ turns out to be the smallest FW kernel for $G$. 

It is interesting to note that a proof of the equality $U = H\cap W = H\cap D$, not using character theory, would give a character-free proof of Wielandt's Theoreom~\ref{wie} for the FW complement $H/U$, and in particular of Frobenius'
theorem.
In fact, by Lemma~\ref{elem}\ref{item:three}, we have $K \subseteq D$. But it is easy to compute the order of the set $K$, and if  $U = H\cap W = H\cap D$, then  we have $|K| = |D|$, whence $K = D$, which is a subgroup.

In \cite{loupass} the quotient $H/U$ is studied, under the name of \emph{generalised Frobenius complement}.

In \cite{espuelas} the following characterisation is given, generalising the characterisation of the Frobenius complements in terms of fixed-point-free linear actions mentioned in the remark after Theorem~\ref{fpf}.

\begin{theorem}\label{esp}

Let $G$ be an FW group, with FW complement $H/H^*$. Then $H/H^*$ is isomorphic to a quotient $L/M$ where the group $L$ is an irreducible subgroup of $\mathrm{GL}(n, q)$ for some $n$, where 
$q$ is a prime not dividing $|L|$, and furthermore $M$ is nilpotent;
and every element in $L \setminus M$ acts fixed-point-freely on $Q = (C_q)^n$.
Moreover, $L$ and $M$ are isomorphic to sections of $H$ and $H^*$ respectively. If, in addition, $H/H^*$ 
is a $p$-group, then $L$ is also a $p$-group.

Conversely, if $L$ is a group acting by automorphisms on a group $Q$, and $M$ is a normal subgroup of $L$ such that every element in $L \setminus M$ acts fixed-point-freely on $Q$, then the semidirect product $LQ$ is an FW group with FW complement $L/M$.

\end{theorem}

\bigskip

Permutation group theorists, on the other side, looking at $G$ acting faithfully by right multiplication on the right cosets of $H$, have been concerned with \emph{derangements},
the elements of the set $\Delta$ above. Note that the subgroup $U$ defined in Definition~\ref{notation} could also be defined as the subgroup of $H$ generated by the elements fixing more than one point in the permutation action, and similarly $W$ is
generated by all elements of $G$ with more than one fixed point.

Jordan~\cite{jordan} showed in 1872
that a transitive finite permutation group $G$ of degree $n>1$
must contain a derangement, and Cameron and Cohen~\cite{cc} showed in 1992
that the number of derangements is at least $|G|/n$. They also characterised
groups in which equality holds: they are just the doubly transitive Frobenius
groups, classified by Zassenhaus~\cite{zassenhaus}, see also~\cite{passman}.
We recommend a paper by Serre~\cite{serre} for some beautiful applications of
these results in number theory and topology.

Bailey \emph{et~al.}~\cite{bcgr}
showed that the index $|G:D|$ is at most $n-1$; and this can be improved to
$\sqrt{n}-1$ if $G$ is primitive but not affine.

Since, as mentioned after Theorem~\ref{frob}, the structure of the Sylow $p$-subgroups of Frobenius complements is restricted, it is reasonable to ask whether there is any restriction on the structure of the Sylow $p$-subgroups of a FW complement. The answer to this question is negative: in~\cite{espuelas} some examples of $p$-groups that are FW complements, but are not even factors of Frobenius complements were presented, while in~\cite{scoppola} it was shown that every $p$-group is isomorphic to the FW complement of a suitable FW group.

\section{Examples of $p$-groups as FW complements}
\label{s:examples}
 
We have established that all groups appearing as $G/D$ are FW complements and, conversely, every FW complement appears as a factor group of a suitable $G/D$.
Note that, in Proposition 3.1 of \cite{bcgr}, a characterisation of $G/D$ is given that is essentially the same as the characterisation of  generalised Frobenius complements that is contained in \cite{loupass}.

Furthermore, note that Theorem~\ref{esp} offers a construction of all FW complements as FW complements of affine type (so that $G$ is an affine group and $H$
the linear group stabilising the zero vector).

For any positive integer $r$, let $F$ be a finitely generated free group, and
$\kappa_{p^r+1}(F)$ its $(p^r+1)$-st modular dimension subgroup.
In \cite{scoppola}
it is shown that the factor group $F/\kappa_{p^r+1}(F)$ can be realised as an
FW complement. The complexity of these $p$-groups grows with the parameter $r$, and every finite $p$-group appears as a factor group of one of them. We do not repeat that construction here, but we focus on 
one of the key results of that paper which is used here. 

\begin{lemma}\label{charH}
Let $H$ be a finite $p$-group and $H^*$ be a proper normal subgroup of $H$. Then the following statements are equivalent.

\begin{itemize}

\item[(a)] There exists a complex $H$-module $M$ such that $H^*$ contains every element of $H$ that fixes a nontrivial vector in $M$.

\item[(b)] There is a cyclic section $C/E$ of $H$ such that every element of $H \setminus H^*$ has a power in $C\setminus E$.

\end{itemize}

\end{lemma}

\begin{proof}

Assume that (a) holds. Without loss of generality assume that $M$ is irreducible. Since $p$-groups are monomial, $M = L^H$, for a linear module $L$ of some subgroup $C$ of $H$, and let $E$ be the kernel of the corresponding linear representation of $C$. Let $\mu$ and $\lambda$ be the characters afforded by $M$ and $L$, respectively. Let $x \in H \setminus H^*$. 
Then by Mackey's Theorem and Frobenius reciprocity,
\begin{eqnarray*}
0 &=& (\mu_{\langle x\rangle},1_{\langle x\rangle})\\
&=& (\lambda^H_{\langle x\rangle},1_{\langle x\rangle})\\
&=& \left(\sum_y (\lambda^y_{C^y \cap\langle x\rangle})^{\langle x\rangle}, 1_{\langle x\rangle}\right)\\
&=& \sum_y\left((\lambda^y_{C^y \cap \langle x\rangle})^{\langle x\rangle}, 1_{\langle x\rangle}\right)\\
&=& \sum_y(\lambda^y_{C^y \cap\langle x\rangle}, 1_{C^y \cap\langle x\rangle}),
\end{eqnarray*}
where $y$ runs over a set of double coset representatives of $C$ and $\langle x\rangle$. In particular, we have $(\lambda_{C \cap\langle x\rangle}, 1_{C \cap\langle x\rangle}) = 0$, and then $C \cap\langle x\rangle \nleq E$.

Conversely, suppose that (b) holds. Induce to $H$ a linear character $\lambda$ of $C$ with kernel $E$. Let $\mu = \lambda ^H$. Since $H^* \normal H$, also $H\setminus H^*$ is a union of conjugacy classes of $H$; therefore all the conjugates of an element $x \in H \setminus H^*$ have a power in $C \setminus E$.
Equivalently $x$ has a power  in $C^y \setminus E^y$ for all $y$, and reading backwards the above equalities we get our result.
\end{proof}

In view of Theorem~\ref{esp}, the above result actually characterises the factors $H/H^*$ that can be realised as FW complements in terms of a property of the group $H$, \emph{independently} of the structure of the group $G$ in which $H$ appears as the stabiliser of a point. On the other hand, it is clear that the structure of subgroups $W$, $U$ and $D$ depends on the structure of $G$.

This is illustrated by the following result.

\begin{theorem}\label{result}

Let  $G$ be a FW group having $H/H^*$ as a FW complement, and $K^*$ as a FW kernel. Assume that $H$ is a finite p-group. Then there is a subgroup $C$ of $H$, and a proper normal subgroup $E$ of $C$, such that
\begin{itemize}
\item $C/E$ is cyclic; and
\item every element of $H \setminus H^*$ has a power in $C \setminus E$.
\end{itemize}
\end{theorem}

\begin{proof} 
We will simply list the steps of the proof, indicating the proper reference in each case. 
By \cite[Lemma 1.1(3)]{espuelas}, $|K^*:H^*|$ is coprime to $p$, and $H$ is a Sylow $p$-subgroup of $G$. Now let $Q$ be a Sylow $q$-subgroup of $K^*$, with $q \neq p$. Clearly $Q$ is also a Sylow subgroup of $G$.  

By the Frattini argument, $N_G(Q)K^* = G$. Thus, up to conjugation, we may assume that $P$ normalizes $Q$. Now setting $U=H\Omega_1(Z(Q))$ in \cite[Lemma 1.2(1)]{espuelas}, we see that $U$ is a FW group with complement $H/H^*$ and FW kernel $H^*\Omega_1(Z(Q))$. Clearly $\Omega_1(Z(Q))$ can be seen as a vector space over the field $F_q$ with $q$ elements, and therefore as a $F_qH$-module. By \cite[Lemma 1.1(1)]{espuelas}, the elements of $H \setminus H^*$ act without fixed points on $\Omega_1(Z(Q))$. But it is well known (see \cite{serre2}, \cite{loupass}) that $H$ can then be seen as a complex linear group, in which the elements of $H \setminus H^*$ do not fix any nontrivial vector. Lemma~\ref{charH} now gives our result.
\end{proof}

Thus while by \cite{scoppola} the structure of $H/H^*$ is completely unrestricted, the structure of $H$ is subject to a significant restriction, in order to allow the construction of a FW complement, while the question about possible restrictions on the structure of $G/D \simeq H/U$ is still open.

In the case of permutation groups we can conclude the following.

\begin{corollary} Let $G$ be a transitive nonregular permutation group, with one-point stabilizer $H$, and assume that $H$ is a finite $p$-group, and $D_G(H) < G$. Then there is a subgroup $C$ of $H$, and a proper normal subgroup $E$ of $C$, such that $C/E$ is cyclic, and every element of $H \setminus H \cap D_G(H)$ has a power in $C \setminus E$.
\end{corollary}

\end{document}